\documentclass[]{amsart}

\usepackage{latexsym}
\usepackage{amssymb,amsmath,amsopn}
\usepackage[dvips]{graphicx}   
\usepackage{wrapfig}
\usepackage{subfig}
\usepackage{color,epsfig}      
\usepackage{url}
\usepackage{amsbsy}

\newtheorem{theorem}{Theorem}[]

\theoremstyle{definition}
\newtheorem{corollary}{Corollary}[]

\newtheorem{remark}[]{Remark}[]

\begin{document}
\title[]{A note on the centres of a closed chain of circles}
\author[\'A. G.Horv\'ath]{\'Akos G.Horv\'ath}
\address {\'A. G.Horv\'ath \\ Department of Geometry \\ Mathematical Institute \\
Budapest University of Technology and Economics\\
H-1521 Budapest\\
Hungary}
\email{ghorvath@math.bme.hu}


\date{May, 2018}

\maketitle

\section{Introduction}

An interesting recent elementary statement on circles is Dao's Theorem on six circles (see \cite{cohl}, \cite{dergiades}, \cite{dung} and \cite{ngo}). This theorem states that if we have a cyclic hexagon and consider six triangles defined by the lines of its three consecutive sides, then the circumcentres of these triangles are the vertices of a Brianchon hexagon, that is a hexagon whose main diagonals are concurrent. Here we note that Brianchon's theorem states that the main diagonals of a hexagon circumscribed to a conic are concurrent. Hence the hexagons circumscribed to a conic are always Brianchon hexagons, see e.g. \cite{stachelbook}). Two consecutive circles of the chain of six circles intersect each other in at most two points; one of them is a vertex of the original hexagon and the other one we call the "second point of intersection" of the two circles. We note that the second points of intersection are not concyclic in a general situation. To see a simple example consider a degenerated Dao's configuration (see Fig. \ref{fig:daocond}). One of the sides of the hexagon has zero length and the corresponding triangle degenerates to a point $I$. We also assume that the circles $k$ and $l$ touch at a point $II$ and similarly the circles $n$ and $p$ touch at a point $III$. The set of second points of intersection contains the points $I$, $II$ and $III$ which determine the original circumscribed circle of the pentagon. Consequently, the second points of intersection cannot be concyclic, as we stated. This shows that the following problem is independent of that of Dao.

\begin{figure}[h]
\includegraphics[scale=0.6]{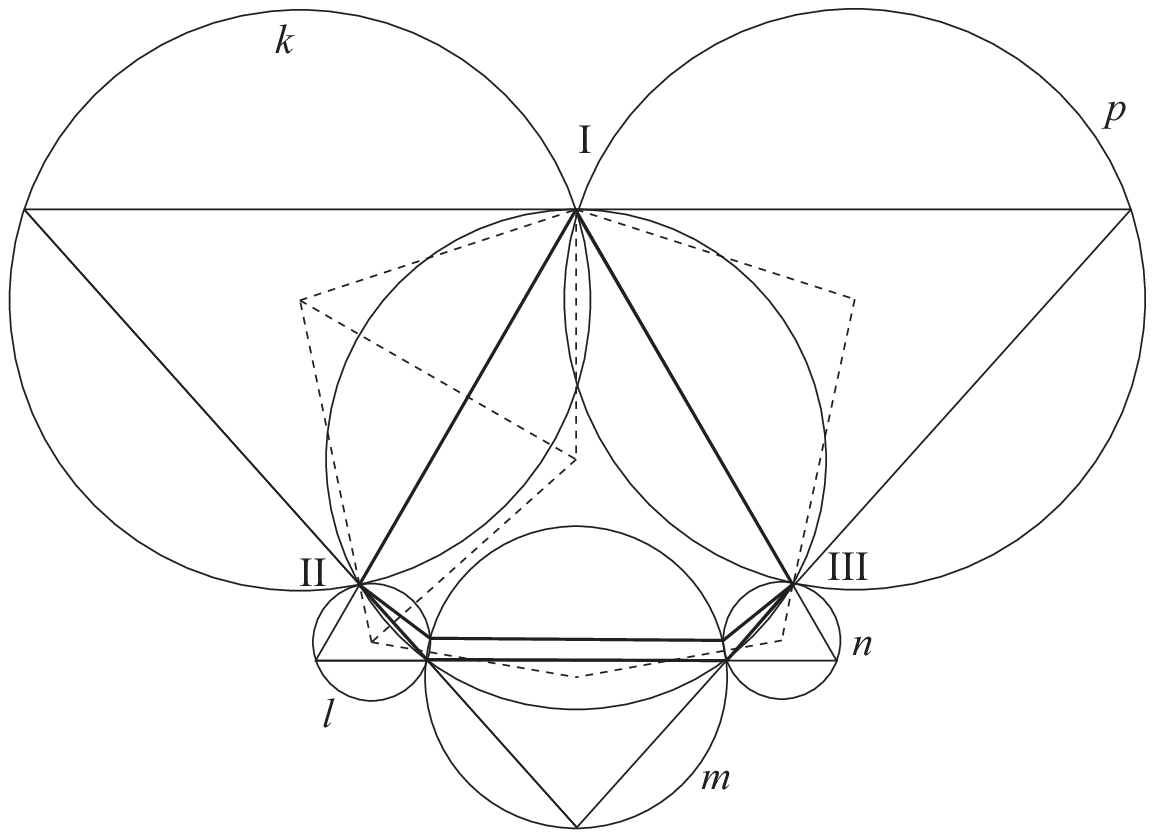}
\caption{The set of "second points" cannot be concyclic in this Dao's configuration.}\label{fig:daocond}
\end{figure}

Miquel's Six-Circles Theorem (see \cite{miquel}) can be formulated in the following way: If we have two cyclic quadrangle $P_1P_2P_3P_4$ and $Q_1Q_2Q_3Q_4$  for which the quadruples $P_1Q_1Q_2P_2$, $P_2Q_2Q_3P_3$, $P_3Q_3Q_4P_4$ are cyclic then the last quadruple of this type $P_4Q_4Q_1P_1$ is also cyclic. The circumcircles of the last four quadruples form a closing chain of intersecting circles with the property that the points of intersection belong to two other circles transversal to each circle of the chain. By induction, on can easily prove the following extension of Miquel's theorem:

\begin{theorem}[\cite{gevay1}]
Let $\alpha$ and $\beta$ be two circles. Let $n>2$ be an even number, and take the points $P_1,\ldots, P_n$ on $\alpha$ and $Q_1,\ldots,Q_n$ on $\beta$, such that each quadruple $P_1Q_1Q_2P_2$, $\ldots, $ $P_{n-1}Q_{n-1}Q_nP_n$ is cyclic. Then the quadruple $P_nQ_nQ_1P_1$ is also cyclic.
\end{theorem}

In the case of $n=6$ the obtained configuration of circles is very similar to the configuration in Dao's theorem, so it is not to surprising that L. Szilassi observed that the centers of the circles form Brianchon hexagon, but he has given no proof. G. G\'evay gave a proof of this statement in \cite{gevay2} using projective geometry of the three-space. On the other hand this problem was published earlier in Crux Mathematicorum by Dao \cite{dao} without solution. It is interesting that in a later volume of Crux Mathematicorum we can find a correction for another problem (see \cite{bataille}) containing the key statement to give a simple solution for the above one.

Our short paper contains a simpler and shorter proof of that the hexagon in question is a Brianchon one. This proof also leads to a generalization of the statement from the case $n=6$ to any even value of $n$.

\section{A theorem on the chain of intersecting circles}

We prove the following theorem:

\begin{theorem}\label{thm:chainofcirc}
Let $c(K)$ and $c(L)$ be two circles with respective centers $K$ and $L$. Let $n>2$ be an even number, and take the points $P_1,\ldots P_n$ on $c(K)$ and $Q_1,\ldots,Q_n$ on $c(L)$, such that each quadruple $P_1Q_1Q_2P_2$, $\ldots $, $P_nQ_nQ_1P_1$ is cyclic. Denote by $O_i$ the center of the circle $c(O_i)$ circumscribed the quadrangle $P_iQ_iQ_{i+1}P_{i+1}$. Then for each value of $i$ the line $O_iO_{i+1}$ is tangent to a fixed conic with foci $K$ and $L$.
\end{theorem}

\begin{proof}
We prove that if we reflect the point $K$ to the successive sides of the polygon $O_1\ldots O_{n}$ then we get points on a circle $c$ with center $L$. Let $T$ denote the reflection of $K$ to the side $O_1O_n$. From the metric definition of a conic it immediately follows that the locus of the reflected image of a focus to the tangent of the conic is a circle which centre is the other focus. From this we obtained that the perpendicular bisector $O_1O_n$ of the segment $KT$ is tangent to a conic with foci $K$ and $L$. (See e.g. \cite{stachelbook}.) Consider the cyclic quadrangle $P_1Q_1Q_2P_2$ (see Fig. \ref{fig:chainofcirc2(1)}).

\begin{figure}[ht]
\includegraphics[scale=0.8]{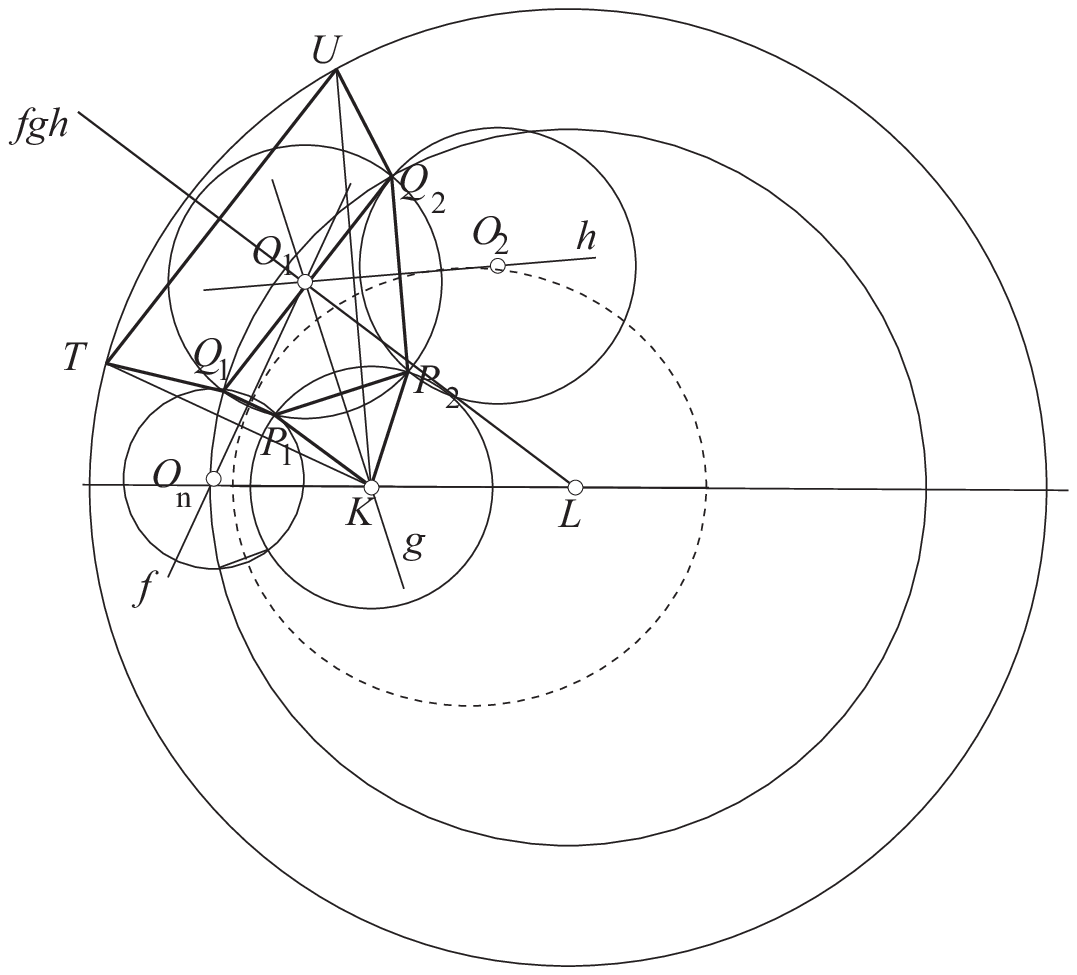}
\caption{The inscribed conic in the case of ellipse.}\label{fig:chainofcirc2(1)}
\end{figure}

Denote by $f$, $g$ and $h$ the reflections in the perpendicular bisectors of $Q_1P_1$, $P_1P_2$ , $P_2Q_2$, respectively. These three lines meet at $O_1$ (or are parallel if $O_1$ is at infinity), so $fgh$ is another reflection. Then $fgh(Q_2)=fg(P_2)=f(P_1)=Q_1$, so $fgh$ is the reflection about the perpendicular bisector of $Q_1Q_2$. But then $fgh(U)=fg(K)=f(K)=T$, so $fgh$ is the reflection about the perpendicular bisector of $TU$, which therefore coincides with the perpendicular bisector of $Q_1Q_2$.
Consequently, the distances of the points $T$ and $U$ from the center $L$ are equal to each other. Hence the conic defined by the foci $K$ and $L$ and the tangent line $O_nO_1$ agree with the conic defined by the same pair of foci and the tangent line $O_1O_2$. The similar reasoning for the next quadrangle $P_2Q_2Q_3P_3$ implies that this conic is the same as the conic defined by the foci $K$, $L$ and the tangent line $O_2O_3$ and so on and so forth. This proves the theorem.
\end{proof}

\begin{corollary}[Theorem 2 in \cite{gevay2}]
If $n=6$ the polygon defined by the centers $O_i$ is a Brianchon hexagon by Brianchon's theorem on conics.
\end{corollary}

\begin{figure}[ht]
\includegraphics[scale=0.6]{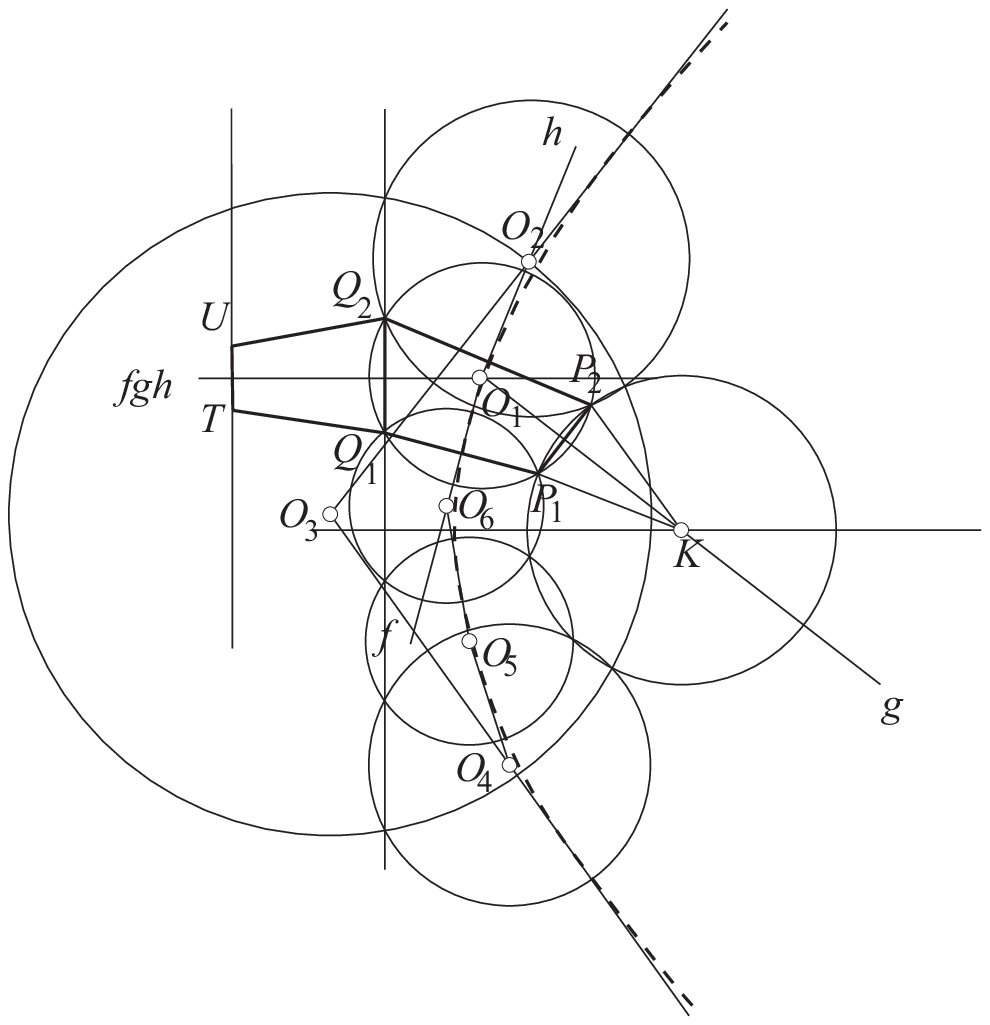}
\caption{The inscribed conic in the case of parabola.}\label{fig:chainofcirc2(2)}
\end{figure}

\begin{remark} Observe that in the case when $K$ is an inner point of the circle $c(L)$ containing the points $Q_i$ the conic is an ellipse and if the point $K$ is an outer point of $c(L)$ then the conic is an hyperbola. The case that the conic is a parabola occurs when the point $L$ is "at infinity" meaning that $c(L)$ is a line. Indeed, in this case the segment $UT$ is parallel to $c(L)$ and the examined lines are tangent to a parabola (see Fig. \ref{fig:chainofcirc2(2)}).
\end{remark}

\end{document}